\documentclass[11pt]{amsart}

\usepackage{amsmath, amsthm, amssymb, amsfonts, amscd, verbatim}

\def\CC{{\mathbb C}}

\def\GG{{\mathbb G}}

\def\QQ{{\mathbb Q}}
\def\PP{{\mathbb P}}
\def\QQ{{\mathbb Q}}
\def\RR{{\mathbb R}}

\def\ZZ{{\mathbb Z}}

\def\Psf{{\mathsf P}}

\def\hhat{{\hat h}}

\def\0{{\mathbf 0}}
\def\1{{\mathbf 1}}

\def\w{{\mathbf w}}

\def\z{{\mathbf z}}

\def\Ecal{{\mathcal E}}
\def\Fcal{{\mathcal F}}

\def\Jcal{{\mathcal J}}

\def\Mcal{{\mathcal M}}

\def\Ocal{{\mathcal O}}
\def\Pcal{{\mathcal P}}

\def\Zcal{{\mathcal Z}}

\def\Kbar{{\bar K}}

\def\diag{\mathrm{diag}}

\def\Gal{\mathrm{Gal}}

\def\supp{\mathrm{supp}}

\theoremstyle{plain}

\newtheorem{thm}{Theorem}

\newtheorem{prop}[thm]{Proposition}
\newtheorem{lem}[thm]{Lemma}

\newtheorem{conj}{Conjecture}

\begin{document}

\title[$S$-Integral Preperiodic Points]{$S$-Integral Preperiodic Points for Dynamical Systems over Number Fields}
\author{Clayton Petsche}
\address{Ph.D. Program in Mathematics; CUNY Graduate Center; 365 Fifth Avenue; New York, NY 10016-4309 U.S.A.}
\email{cpetsche@gc.cuny.edu}
\date{September 11th, 2007}
\begin{abstract}
Given a rational map $\phi:\PP^1\to\PP^1$ defined over a number field $K$, we prove a finiteness result for $\phi$-preperiodic points which are $S$-integral with respect to a non-preperiodic point $P$, provided $P$ satisfies a certain local condition at each place.  This verifies a special case of a conjecture of S. Ih.
\end{abstract}

\maketitle


\section{Introduction}

Let $K$ be a number field with algebraic closure $\Kbar$, let $S$ be a finite set of places of $K$ including all of the archimedean places, and let $\Ocal_S$ denote the ring of $S$-integers in $K$.  Given two points $P,Q\in\PP^1(\Kbar)$, we say that $Q$ is $S$-integral with respect to $P$ if the Zariski closures of $P$ and $Q$ do not meet in the model $\PP^1_{\Ocal_S}$ for $\PP^1$.  Intuitively, this means that $P$ and $Q$ have distinct reduction with respect to each place of $\Kbar$ lying over the places of $K$ outside $S$.

Let $\phi:\PP^1\to\PP^1$ be a rational map of degree at least two defined over $K$.  A point $P\in\PP^1(\Kbar)$ is called periodic (with respect to $\phi$) if $\phi^n(P)=P$ for some $n\geq1$, where $\phi^n$ denotes the map $\phi$ composed with itself $n$ times.  More generally, $P$ is called preperiodic if some iterate $\phi^n(P)$ of it is periodic.  By analogy with several well-known results in Diophantine geometry, S. Ih has conjectured the following finiteness property for $S$-integral preperiodic points.

\begin{conj}[Ih]\label{IhConj}
Let $K$ be a number field, let $\phi:\PP^1\to\PP^1$ be a rational map of degree at least two defined over $K$, and let $S$ be a finite set of places of $K$ including all of the archimedean places.  Let $P\in\PP^1(\Kbar)$ be a non-preperiodic point.  Then the set of preperiodic points in $\PP^1(\Kbar)$ which are $S$-integral with respect to $P$ is finite.
\end{conj}

Baker-Ih-Rumely \cite{BakerIhRumely} have proved this conjecture in the following cases: {\bf (a)} $\phi(z)=z^n$ for some $n\geq2$, where $z$ denotes the standard affine coordinate on $\PP^1$;  {\bf (b)} $\phi$ is a Latt\'es map associated to an elliptic curve $E/K$;  and {\bf (c)} $\phi$ is a Chebyshev map (although in this latter case they do not work out the details in \cite{BakerIhRumely}).  In all three of these cases the map $\phi$ is defined via the multiplication-by-$n$ map on a commutative algebraic group $G$ (either $\GG_m$ or $E$), in the latter two cases by replacing $G$ with a quotient of itself by some nontrivial automorphism.  The $\phi$-preperiodic points on $\PP^1$ are thus closely related to the torsion points on the group $G$.

In this paper we will prove a special case of Conjecture \ref{IhConj} which holds for an arbitrary map $\phi$, but which requires certain local conditions on the non-preperiodic point $P\in\PP^1(\Kbar)$.  For each place $v$ of $K$, the map $\phi:\PP^1(\CC_v)\to\PP^1(\CC_v)$ determines a decomposition $\PP^1(\CC_v)=\Fcal_v(\phi)\cup\Jcal_v(\phi)$ of the projective line into two disjoint subsets, the the $v$-adic Fatou set $\Fcal_v(\phi)$ and Julia set $\Jcal_v(\phi)$.  We will recall the definitions of these two sets in $\S$\ref{FatJulSect}, but intuitively one thinks of the dynamical behavior of $\phi$ as being ``stable'' on the Fatou set, and ``chaotic'' on the Julia set.  Given a point $P\in\PP^1(\Kbar)$ and a place $v$ of $K$, we say that $P$ is a totally Fatou point at $v$ (with respect to $\phi$) if each of the $[K(P):K]$ embeddings of $P$ into $\PP^1(\CC_v)$ is contained in the Fatou set $\Fcal_v(\phi)$.

\begin{thm}\label{MainThm}
Let $K$, $S$, $\phi:\PP^1\to\PP^1$, and $P\in\PP^1(\Kbar)$ satisfy the same hypotheses as in Conjecture \ref{IhConj}.  Assume in addition that $P$ is a totally Fatou point at all places $v$ of $K$.  Then the set of preperiodic points $\PP^1(\Kbar)$ which are $S$-integral with respect to $P$ is finite.
\end{thm}

If $v$ is an archimedean place, then there are examples of rational maps $\phi$ for which the Fatou set $\Fcal_v(\phi)$ is empty; in this case it is impossible for any point to be totally Fatou at $v$.  However, in all other cases $\Fcal_v(\phi)$  is ``large'' in the sense that it is an open, dense subset of $\PP^1(\CC_v)$, and thus there is a sense in which totally Fatou points are generic.  Moreover, if $v$ is non-archimedean and $\phi$ has good reduction at $v$, then $\Fcal_v(\phi)=\PP^1(\CC_v)$.  In particular, every point in $\PP^1(\Kbar)$ is totally Fatou at every place of $K$ outside of a finite set of places depending only on $\phi$.

\medskip

{\em Example 1.}  Let $K=\QQ$, let $S$ contain only the archimedean place, and identify $\PP^1(\bar{\QQ})=\bar{\QQ}\cup\{\infty\}$ by the standard affine coordinate $z=(z:1)$, where $\infty=(1:0)$.  In this case the set of $S$-integral points with respect to $\infty$ is precisely the set $\bar{\ZZ}$ of algebraic integers in $\bar{\QQ}$.  Define $\phi:\PP^1\to\PP^1$ by $\phi(z)=z^2/(z^2+4z+1)$.  The point $\infty$ is not preperiodic, and at the archimedean place it is attracted to the super-attracting fixed point $0$; hence $\infty$ is in the archimedean Fatou set.  Since $\phi$ has good reduction at all non-archimedean places, every point (in particular $\infty$) is Fatou at those places.  Theorem \ref{MainThm} implies that only finitely many elements of $\bar{\ZZ}$ are $\phi$-preperiodic.  It is easy to see that $\phi$ is not covered by the results in \cite{BakerIhRumely}, and thus gives a genuinely new example of the finiteness property predicted by Conjecture \ref{IhConj}.

\medskip

{\em Example 2.}  With $K=\QQ$ and $S$ as above, consider any monic polynomial $\phi(z)\in\ZZ[z]$ of degree $d\geq2$; thus $\infty$ is fixed by $\phi$.  All of the (infinitely many) preperiodic points except $\infty$ are in $\bar{\ZZ}$, since any root of $\phi^{m}(z)-\phi^n(z)=0$ for $0\leq n<m$ is an algebraic integer.  This shows that in Conjecture \ref{IhConj}, the requirement that $P$ is not preperiodic cannot be omitted. 

\medskip

{\em Example 3.}  With $K=\QQ$ and $S$ as above, let $\phi(z)=2z^2/(4z^2+7z+1)$.  The point $\infty$ is not preperiodic, and is Fatou at the archimedean place and at all odd primes.  However, $\infty$ is in the $2$-adic Julia set $\Jcal_2(\phi)$.  Thus, while Conjecture \ref{IhConj} predicts that only finitely many elements of $\bar{\ZZ}$ are $\phi$-preperiodic, our local methods are not enough to deduce this.  [To check the claims made here, it helps to know that $\phi=f^{-1}\circ\psi\circ f$, where $\psi(z)=(z^2-z)/2$ and $f(z)=(4z+1)/z$, and that $\Jcal_2(\psi)=\ZZ_2$.] 

\medskip

The primary tool in the proof of Theorem \ref{MainThm} is the equidistribution theorem for small points with respect to the canonical height $\hhat_\phi$ associated to $\phi$, which for each place $v$ of $K$ governs the limiting distribution of such points on the  Berkovich projective line $\Psf^1_v$ associated to $\PP^1(\CC_v)$.  Following the initial approach of Baker-Ih-Rumely \cite{BakerIhRumely}, we assume on the contrary that the set of preperiodic points which are $S$-integral with respect to $P$ is infinite.  The idea is to then use the equidistribution theorem at each place $v\in S$ to test these preperiodic points against a certain function $F$ on $\Psf^1_v$, which is defined as a sum of Arakelov-Green's functions, and to achieve a contradiction to the assumption that $P$ is not preperiodic.

However, this approach is complicated by the fact that the function $F$ has logarithmic singularities at the embeddings of $P$ into $\Psf_v^1$, and thus the equidistribution theorem does not apply directly.  In the special cases arising from algebraic groups treated in \cite{BakerIhRumely}, Baker-Ih-Rumely get around this problem by replacing $F$ with a truncation $F^*$ depending on other parameters, and combining: (a) a global Diophantine ingredient (linear forms in logarithms) to ensure that preperiodic points do not get too close to the point $P$ in terms of their heights and degrees; and (b) a quantitative version of the equidistribution theorem.

In this paper we use the totally Fatou hypothesis to prove the finiteness result using purely local means.  Our proof makes essential use of the following two properties of the Fatou set $\Fcal_v(\phi)$: (a) the preperiodic points are discrete in $\Fcal_v(\phi)$ (Proposition \ref{DiscThm}), and (b) the target measure in the equidistribution theorem is supported outside $\Fcal_v(\phi)$.  These two facts allow the truncation argument to proceed without the use of linear forms in logarithms, and without a quantitative equidistribution theorem.

Finally, we remark that if $v$ is a place at which the Julia set $\Jcal_v(\phi)$ is nonemtpy, then it is known (due to Fatou and Julia in the archimedean case, and Hsia in the non-archimedean case) that each point of $\Jcal_v(\phi)$ is a limit point of preperiodic points.  Therefore, Theorem \ref{MainThm} is in some sense the most general result toward Conjecture \ref{IhConj} which can be proved using purely local methods; any further progress should require the use of some global Diophantine arguments.

\medskip

{\em Acknowledgment.}  The author thanks Rob Benedetto for pointing out Rivera-Letelier's results and their application to the proof of Proposition \ref{DiscThm} in the non-archimedean case.


\section{Rational Dynamics Over Local Fields}\label{LocDynSect}

\subsection{Local analytic preliminaries}

Let $v$ be a place of the number field $K$, let $|\cdot|_v$ be an absolute value associated to $v$, and let $\CC_v$ be the completion of the algebraic closure of the completion of $K$ at $v$; thus $\CC_v$ is the smallest algebraically closed extension of $K$ which is complete with respect to $|\cdot|_v$.  Define a norm $\|\cdot\|_v$ on $\CC_v^2$ by $\|\z\|_v=(|z_0|_v^2+|z_1|_v^2)^{1/2}$ if $v$ is archimedean, and by $\|\z\|_v=\max\{|z_0|_v,|z_1|_v\}$ if $v$ is non-archimedean, where $\z=(z_0,z_1)\in\CC_v^2$.  The projective metric $\delta_v(\cdot,\cdot)$ on $\PP^1(\CC_v)$ with respect to this absolute value is defined by
\begin{equation*}
\delta_v(P,Q) = \frac{|\z\wedge\w|_v}{\|\z\|_v\|\w\|_v},
\end{equation*}
where $\z,\w\in\CC_v^2$ are lifts of $P,Q\in\PP^1(\CC_v)$ respectively, and where $\z\wedge\w=z_0w_1-z_1w_0$.  For each $P\in\PP^1(\CC_v)$ and $r>0$, we denote the open disc of radius $r$ about $P$ by $D_r(P) = \{Q\in\PP^1(\CC_v)\mid \delta_v(P,Q)<r\}$.

\subsection{Reduction}  Suppose that $v$ is a non-archimedean place of $K$.  We let $\Ocal_v$ denote the ring of integers in $\CC_v$, $\Mcal_v$ its maximal ideal, and $k_v=\Ocal_v/\Mcal_v$ the residue field.  Let $\PP_{\Ocal_v}^1$ denote the usual integral model for $\PP^1$, and let $r_v:\PP^1(\CC_v)\to\PP^1(k_v)$ be the induced reduction map.  It is straightforward to show that given two points $P,Q\in\PP^1(\CC_v)$, we have $r_v(P)=r_v(Q)$ if and only if $\delta_v(P,Q)<1$.

A rational map $\phi:\PP^1\to\PP^1$ defined over $\CC_v$ is said to have good reduction if it extends to a morphism $\phi:\PP_{\Ocal_v}^1\to\PP_{\Ocal_v}^1$ of schemes.  In particular, such a map induces a a rational map $\tilde{\phi}:\PP^1\to\PP^1$ defined over the residue field $k_v$.

\subsection{The Fatou and Julia sets}\label{FatJulSect}

Let $\phi:\PP^1\to\PP^1$ be a rational map defined over $\CC_v$, and let $P\in\PP^1(\CC_v)$ be a point.  We say that the family of iterates $\{\phi^n\}_{n=1}^{\infty}$ is equicontinuous at $P$ if for each $\epsilon>0$ there exists some $\delta>0$ such that $\phi^n(D_\delta(P))\subseteq D_\epsilon(\phi^n(P))$ for all $n\geq1$.  The Fatou set $\Fcal_v(\phi)$ associated to $\phi$ is defined as the largest open subset of $\PP^1(\CC_v)$ such that $\{\phi^n\}_{n=1}^{\infty}$ is equicontinuous at each point $P\in \Fcal_v(\phi)$.  The Julia set $\Jcal_v(\phi)$ is defined to be the complement in $\PP^1(\CC_v)$ of the Fatou set.  

By definition the Fatou set is open and the Julia set is closed.  The Fatou set (and therefore also the Julia set) is completely invariant with respect to $\phi$ in the sense that $\phi^{-1}(\Fcal_v(\phi))=\Fcal_v(\phi)=\phi(\Fcal_v(\phi))$.  Moreover, if $v$ is a non-archimedean place at which $\phi$ has good reduction, then $\Jcal_v(\phi)=\emptyset$.  For more details and for the proofs of the basic facts about these sets, see \cite{Silverman} $\S$1.4 and $\S$5.4.

\subsection{The discreteness of preperiodic points in the Fatou set}

In this section we will prove the following result on Fatou preperiodic points.

\begin{prop}\label{DiscThm}
Let $v$ be a place of $K$, and let $\phi:\PP^1\to\PP^1$ be a rational map of degree $d\geq2$ defined over $\CC_v$.  Then the set of preperiodic points in $\Fcal_v(\phi)$ is discrete in $\Fcal_v(\phi)$.
\end{prop}

If $v$ is archimedean the proposition follows fairly easily from the finiteness of Fatou periodic points.  If $v$ is non-archimedean this finiteness property fails, which makes the proof somewhat more complicated in this case.  We first record the following lemma.

\begin{lem}\label{ArchDiscLemma}
Let $v$ be a place of $K$, and let $\phi:\PP^1\to\PP^1$ be a rational map of degree $d\geq2$ defined over $\CC_v$.  Let $P$ be a point in the Fatou set $\Fcal_v(\phi)$ which is preperiodic but not periodic.  Given any $Q\in\PP^1(\CC_v)$, $P$ is not a limit point of the sequence $\{\phi^n(Q)\}$.
\end{lem}
\begin{proof}
We will first show that there exists some $\delta>0$ such that the disc $D_\delta(P)$ does not meet any of its images $\phi^n(D_\delta(P))$ under iterates of $\phi$, where $n\geq1$.  To see this, choose $\epsilon>0$ so small that for each point $P'$ in the forward orbit $O^+(P)=\{\phi^n(P) \mid n\geq1\}$ of $P$, $D_\epsilon(P)$ does not meet $D_\epsilon(P')$; this is possible since by assumption $O^+(P)$ is finite and does not contain $P$.  By the equicontinuity of $\{\phi^n\}$ there exists $\delta>0$ such that $\phi^n(D_\delta(P))\subseteq D_\epsilon(\phi^n(P))$ for all $n\geq1$.  Decreasing $\delta$ if necessary we may assume $\delta<\epsilon$, and none of the discs $D_\epsilon(\phi^n(P))$ meet $D_\epsilon(P)\supseteq D_\delta(P)$; thus $D_\delta(P)\cap\phi^n(D_\delta(P))=\emptyset$ for all $n\geq1$, completing the proof of the claim.

Now, it follows from the existence of such $\delta>0$ that given any $Q\in\PP^1(\CC_v)$, at most one term in the sequence $\{\phi^n(Q)\}$ can fall into the set $D_\delta(P)$.
\end{proof}

\begin{proof}[Proof of Proposition \ref{DiscThm} when $v$ is archimedean.]
Let $E$ be the set of points $P$ in $\Fcal_v(\phi)$ such that $P$ is not periodic but $\phi(P)$ is periodic.  Plainly $E$ is finite since $\Fcal_v(\phi)$ contains only finitely many periodic points (\cite{Silverman} Thm. 1.35 (a)), and each point has at most $d$ pre-images.  By the complete invariance of the Fatou set, the set of preperiodic points in $\Fcal_v(\phi)$ is precisely the union of the finite set $E$, its backward orbit 
\begin{equation*}
O^-(E)=\{Q\in\PP^1(\CC_v)\mid \phi^n(Q)\in E\text{ for some }n\geq1\},
\end{equation*}
and the finite set of periodic points in $\Fcal_v(\phi)$, In particular, in order to prove the proposition it suffices to show that $O^-(E)$ is discrete in $\Fcal_v(\phi)$.

Suppose on the contrary that $O^-(E)$ has a limit point $L$ in $\Fcal_v(\phi)$.  Then there exists a sequence of distinct points $\{Q_n\}$ in $O^-(E)$ converging to $L$.  Since $E$ is finite, passing to a subsequence we may assume that all points of the sequence $\{Q_n\}$ are elements of $O^-(P)$ for a single point $P\in E$.  It follows that $P$ is a limit point of the sequence $\{\phi^n(L)\}$.  [Proof: let $\epsilon>0$ be arbitrary.  By the equicontinuity property there exists $\delta$ such that $\phi^m(D_\delta(L))\subseteq D_\epsilon(\phi^m(L))$ for all $m\geq1$.  Select an element $Q_n$ in our sequence $\{Q_n\}$ which is in the disc $D_\delta(L)$; then  $P\in D_\epsilon(\phi^m(L))$ for the $m\geq1$ satisfying $\phi^m(Q_n)=P$.  Thus each disc $D_\epsilon(P)$ around $P$ contains some iterate $\phi^m(L)$ of $L$.]  The fact that $P$ is a limit point of the sequence $\{\phi^n(L)\}$ contradicts Lemma \ref{ArchDiscLemma}, and thus completes the proof that $O^-(E)$ is discrete in $\Fcal_v(\phi)$.
\end{proof}

When $v$ is non-archimedean, the Fatou set $\Fcal_v(\phi)$ may contain infinitely many periodic points, and therefore a direct translation of the archimedean proof of Proposition \ref{DiscThm} is not possible.  It turns out, however, that the discreteness of preperiodic points follows almost immediately from results in J. Rivera-Letelier's thesis, published as \cite{RiveraLetelier}.  Actually a related result which treats only periodic points is stated as Cor. 4.28 in \cite{RiveraLetelier}; the extra arguments needed to treat preperiodic points are relatively straightforward.  

We will now briefly summarize the needed results of Rivera-Letelier, and we will then complete the proof of Proposition \ref{DiscThm}.  Let $v$ be a non-archimedean place of $K$.  Rivera-Letelier defines a certain subset $\Ecal_v(\phi)$ of $\Fcal_v(\phi)$ called the quasiperiodicity domain, which can be characterized as the largest open subset of $\PP^1(\CC_v)$ such that each point $P\in\Ecal_v(\phi)$ is a limit point of its forward orbit $\{\phi^n(P)\}_{n=1}^{\infty}$.  One of the main results of \cite{RiveraLetelier} (see Th\'eor\`em de Classification, $\S$4.4) states that each point $P$ in the Fatou set $\Fcal_v(\phi)$ has one of the following three properties:
\begin{itemize}
	\item $P$ is an element of a wandering disc $D=D_r(P)$ for some $r>0$; this means that $\phi^m(D)\cap\phi^n(D)=\emptyset$ for all $0\leq m<n$. 
	\item $P$ is an element of the basin of attraction $B(Q)=\{R\in\PP^1(\CC_v) \mid \phi^{kn}(R)\to Q \text{ as }n\to+\infty\}$ of an attracting periodic point $Q$ of period $k\geq1$.  
	\item Some iterate $\phi^n(P)$ of $P$ is an element of $\Ecal_v(\phi)$.
\end{itemize}

\begin{proof}[Proof of Proposition \ref{DiscThm} when $v$ is non-archimedean]

Suppose on the contrary that $L\in\Fcal_v(\phi)$ is a limit point of preperiodic points.  Plainly $L$ cannot be an element of a wandering disc $D$, since such discs contain no preperiodic points.

Suppose that $L$ is an element of the basin of attraction $B(Q)$ of an attracting periodic point $Q$.  By replacing $\phi$ with an appropriate iterate $\phi^k$ we may assume without loss of generality that $Q$ is fixed by $\phi$; thus $Q$ is the only periodic point in $B(Q)$.  Let $E=\phi^{-1}(Q)\setminus\{Q\}$.  If $E$ is empty then $Q$ is the only preperiodic point in $B(Q)$, which contradicts the assumption that $L\in B(Q)$ is a limit point of preperiodic points.  Plainly every preperiodic point in $B(Q)$ is an element of $\{Q\}\cup E\cup O^-(E)$, and therefore since $E$ is finite we may assume without loss of generality that $L$ is a limit point of $O^-(P)$ for some $P\in E$.  It follows from the same argument (using equicontinuity) we used in the archimedean case that $P$ is a limit point of the sequence $\phi^n(L)$; this contradicts Lemma \ref{ArchDiscLemma}, since $P$ is preperiodic but not periodic.  Thus, $L$ cannot be an element of an attracting basin.

Finally, by Rivera-Letelier's classification theorem, the only remaining possibility is that some iterate $\phi^n(L)$ is an element of the quasiperiodicity domain $\Ecal_v(\phi)$.  Replacing $L$ with $\phi^n(L)$ we may assume without loss of generality that $L\in\Ecal_v(\phi)$.  But the set of preperiodic points in $\Ecal_v(\phi)$ is discrete: Rivera-Letelier shows (\cite{RiveraLetelier} Prop. 3.16) that there exists an analytic function on $\Ecal_v(\phi)$ which vanishes precisely on the periodic points of $\Ecal_v(\phi)$, and since $\phi$ is injective on $\Ecal_v(\phi)$ (\cite{RiveraLetelier} Prop. 3.9 (ii)), $\Ecal_v(\phi)$ contains no preperiodic points which are not periodic.  We conclude that the set of preperiodic points can have no limit point in $\Fcal_v(\phi)$. 
\end{proof}

\subsection{The Berkovich projective line and the canonical measure}\label{BerkLine}

For each place $v$ of $K$, the Berkovich projective line $\Psf_v^1$ is a Hausdorff, path-connected, and compact topological space which contains the ordinary projective line $\PP^1(\CC_v)$ as a dense subspace.  In the archimedean case $\Psf_v^1=\PP^1(\CC_v)$, but when $v$ is nonarchimedean $\Psf_v^1$ contains many points in addition to the ones in $\PP^1(\CC_v)$.  The $v$-adic metric $\delta_v(\cdot,\cdot)$ extends naturally to $\Psf_v^1$, although in the non-archimedean case it does not define a metric on the larger space, since for example $\delta_v(P,P)>0$ for $P\in \Psf_v^1\setminus \PP^1(\CC_v)$.  For the definition and basic properties of the Berkovich projective line see \cite{BakerRumelyBook} or \cite{Berkovich}.

Any rational map $\phi:\PP^1\to\PP^1$ defined over $\CC_v$ extends naturally to a map $\phi:\Psf_v^1\to\Psf_v^1$.  Assuming $\phi$ has degree $d\geq2$, the canonical measure $\mu_{\phi,v}$ associated to $\phi$ is a certain positive unit Borel measure on $\Psf_v^1$ which satisfies the identities $\phi_*(\mu_{\phi,v}) = \mu_{\phi,v}$ and $\phi^*(\mu_{\phi,v}) = d\cdot\mu_{\phi,v}$.  There are several equivalent definitions of this measure; for example in \cite{BakerRumelyBook} it is defined as (essentially) the negative Laplacian of a Call-Silverman local canonical height function associated to $\phi$.  In our proof of Theorem \ref{MainThm}  we will make essential use of the fact that $\mu_{\phi,v}$ occurs as the limiting measure in the equidistribution theorem for dynamically small points in $\PP^1(\Kbar)$; see $\S$\ref{EquidistSect} for the statement of this theorem.  Finally, we will also need to use the fact that the support of the canonical measure in $\Psf^1_v$ does not meet the Fatou set; that is
\begin{equation}\label{SuppFatou}
\supp(\mu_{\phi,v})\cap\Fcal_v(\phi)=\emptyset.
\end{equation}
For more details on the canonical measure see \cite{FreireLopesMane} for the archimedean case and \cite{BakerRumelyBook}, Ch. 8 for the non-archimedean case.


\section{Heights, Equidistribution, and the Proof of Theorem \ref{MainThm}}

\subsection{Normalized absolute values and global heights}

Let $K$ be a number field, let $M_K$ denote the set of places of $K$, and for each $v\in M_K$ select an absolute value $|\cdot|_v$ associated to $v$ such that the product formula holds in the form $\prod_{v\in M_K}|\alpha|_v=1$ for all $\alpha\in K^\times$.  Then more generally, given a finite extension $L/K$, we have $\prod_{v\in M_K}\prod_{\sigma}|\sigma(\alpha)|_v=1$ for all $\alpha\in L$, where the inner product is taken over all $[L:K]$ $K$-embeddings $\sigma:L\hookrightarrow\CC_v$.  The absolute Weil height function $h:\PP^1(\Kbar)\to\RR$ is defined by 
\begin{equation*}
h(P) = \frac{1}{[L:K]}\sum_{v\in M_K}\sum_{\sigma:L\hookrightarrow\CC_v}\max\{|\sigma(z_0)|_v,|\sigma(z_1)|_v\},
\end{equation*}
where $L/K$ is a finite extension such that $P\in \PP^1(L)$, and $\z=(z_0,z_1)\in L^2$ is a lift of $P$; as is well-known, this definition does not depend on the choice of $L$ or $\z$.

Let $\phi:\PP^1\to\PP^1$ be a rational map of degree $d\geq2$ defined over $K$.  The Call-Silverman canonical height function $\hhat_\phi:\PP^1(\Kbar)\to\RR$ associated to $\phi$ is defined by $\hhat_\phi(P)=\lim_{n\to\infty}\frac{1}{d^n}h(\phi^n(P))$.  A basic property of the canonical height is that $\hhat_\phi(P)\geq0$ for all $P\in\PP^1(\Kbar)$, with equality if and only if $P$ is preperiodic; see \cite{Silverman} Thm. 3.22.

\subsection{Dynamical Arakelov-Green's functions}

Like many global height functions, the canonical height $\hhat_\phi$ can be decomposed into a sum of local terms.  A useful object for this purpose is the family of dynamical Arakelov-Green's functions
\begin{equation*}
g_{\phi,v} : \Psf_v^1\times \Psf_v^1\setminus\diag\{\PP^1(\CC_v)\}\to\RR,
\end{equation*}
defined for each place $v$ of $K$; here $\diag\{\PP^1(\CC_v)\} = \{(P,P)\in\Psf_v^1\times \Psf_v^1 \mid P\in \PP^1(\CC_v)\}$.  The function $g_{\phi,v}$ is symmetric, continuous, and has a logarithmic singularity along $\diag\{\PP^1(\CC_v)\}$ in the sense that $(P,Q)\mapsto g_{\phi,v}(P,Q)+\log\delta_v(P,Q)$ extends to a continuous function on $\Psf_v^1\times\Psf_v^1$.  As a function of one variable $P\in\PP^1(\CC_v)$, $g_{\phi,v}$ can be viewed as a continuously varying family of Call-Silverman canonical local height functions in the sense of \cite{CallSilverman}.  For the definition and the proofs of these basic properties of $g_{\phi,v}$ see \cite{BakerRumely} $\S$3.4.  

In addition to its relationship with the canonical height, the function $g_{\phi,v}$ has the following two properties which will be useful to us: first, as a function of each variable, $g_{\phi,v}$ is orthogonal to the equilibrium measure $\mu_{\phi,v}$; and second, if $v$ is a place at which $\phi$ has good reduction, and if $P,Q\in\PP^1(\CC_v)$, then $g_{\phi,v}(P,Q)$ is closely related to the intersection multiplicity of $P$ and $Q$ in $\PP^1_{\Ocal_v}$.  In the following proposition we will give more precise statements of the aforementioned properties of the functions $\hhat_\phi$ and $g_{\phi,v}$.

\begin{prop}\label{HeightProp}
Let $\phi:\PP^1\to\PP^1$ be a rational map of degree $d\geq2$ defined over a number field $K$.  
\begin{quote}
{\bf (a)}  $\int g_{\phi,v}(P,Q)d\mu_{\phi,v}(P)=0$ for each place $v$ of $K$ and each $Q\in\PP^1(\CC_v)$. \\
{\bf (b)}  Suppose that $v$ is non-archimedean place at which $\phi$ has good reduction, and let $P,Q\in\PP^1(\Kbar)$ be distinct points.  Then $g_{\phi,v}(P,Q)=-\log \delta_v(P,Q)$.  In particular, $g_{\phi,v}(P,Q)\geq0$, with equality if and only if $P$ and $Q$ do not meet in $\PP^1_{\Ocal_v}$.  \\
{\bf (c)}  Let $L/K$ be a finite extension, and let $P,Q\in\PP^1(L)$ be distinct points.  Then  
\begin{equation}\label{LocalGlobalHeight}
\hhat_\phi(P) + \hhat_\phi(Q) = \frac{1}{[L:K]}\sum_{v\in M_K}\sum_{\sigma:L\hookrightarrow\CC_v}g_{\phi,v}(\sigma(P),\sigma(Q)).
\end{equation}
\end{quote}
\end{prop}
\begin{proof}
{\bf (a)}  \cite{BakerRumely} Cor. 4.13.  {\bf (b)}  \cite{BakerRumely} Example 3.43.  {\bf (c)}  \cite{BakerRumely} Lemma 4.1.
\end{proof}

{\em Remark.}  There is no issue of convergence in the right-hand-side of $(\ref{LocalGlobalHeight})$, as all but finitely many summands are zero.

\subsection{The equidistribution theorem and the proof of Theorem \ref{MainThm}}\label{EquidistSect}

The following theorem states that, at a given place $v$ of $K$, small global points tend to equidistribute with respect to the canonical measure $\mu_{\phi,v}$ on the Berkovich projective line $\Psf^1_v$.  When $v$ is archimedean this generalizes results of Bilu \cite{Bilu} and Szpiro-Ullmo-Zhang \cite{SzpiroUllmoZhang}; at the level of generality presented here the result is due independently to several authors; see below.

Given a finite set $\Zcal$ of points in $\PP^1(\Kbar)$, define its height to be the average $\hhat_\phi(\Zcal)=(1/|\Zcal|)\sum_{Z\in\Zcal}\hhat_\phi(Z)$ of the height of its points.

\begin{thm}[Baker-Rumely \cite{BakerRumely}, Chambert-Loir \cite{ChambertLoir}, Favre-Rivera-Letelier \cite{FavreRiveraLetelier}]\label{EquidistThm}
Let $K$ be a number field, and let $\phi:\PP^1\to\PP^1$ be a rational map of degree at least two defined over $K$.  Let $v$ be a place of $K$, and let $\epsilon:\Kbar\hookrightarrow\CC_v$ be a $K$-embedding.  Let $\{\Zcal_n\}_{n=1}^{\infty}$ be a sequence of finite $\Gal(\Kbar/K)$-stable subsets of $\PP^1(\Kbar)$ with $\hhat_\phi(\Zcal_n)\to0$ and $|\Zcal_n|\to+\infty$.  Then
\begin{equation*}
\lim_{n\to+\infty}\frac{1}{|\Zcal_n|}\sum_{Z\in \Zcal_n}F(\epsilon(P))=\int F d\mu_{\phi,v}
\end{equation*}
for all continuous functions $F:\Psf^1_v\to\RR$.
\end{thm}

Finally, we are ready to prove our main result.

\begin{proof}[Proof of Theorem \ref{MainThm}]

Enlarging the set $S$ only enlarges the set of $S$-integral points with respect to $P$, so we may assume without loss of generality that $S$ contains all places over which $\phi$ has bad reduction.  Suppose contrary to the statement of the theorem that the set of preperiodic $S$-integral points with respect to $P$ is infinite; then there exists a sequence $\{\Zcal_n\}$ of $\Gal(\Kbar/K)$-stable sets of such points with $|\Zcal_n|\to+\infty$.  Let $\Pcal$ denote the set of $\Gal(\Kbar/K)$-conjugates of $P$, and for each place $v$ of $K$, fix once and for all an embedding $\epsilon_v:\Kbar\hookrightarrow\CC_v$.

For each place $v$ of $K$ and each $n\geq1$ define
\begin{equation*}
\Gamma_v(n) = \frac{1}{|\Zcal_n||\Pcal|}\sum_{Z\in\Zcal_n}\sum_{Q\in \Pcal}g_{\phi,v}(\epsilon_v(Z),\epsilon_v(Q)),
\end{equation*}
and let $\Gamma(n)=\sum_{v\in M_K}\Gamma_v(n)$.  We are going to show on the one hand that 
\begin{equation}\label{Contradiction1}
\Gamma(n) = \hhat_\phi(P) \text{ for all } n\geq1,
\end{equation}
and on the other hand that
\begin{equation}\label{Contradiction2}
\lim_{n\to+\infty}\Gamma(n) = 0.
\end{equation}
Since $P$ is not preperiodic, $\hhat_\phi(P)>0$, and therefore $(\ref{Contradiction1})$ and $(\ref{Contradiction2})$ contradict one another, which proves the theorem.

To prove $(\ref{Contradiction1})$, fix $n\geq1$, and let $L/K$ be a finite Galois extension which is large enough so that $\Zcal_n\cup\Pcal\subset\PP^1(L)$.  We have 
\begin{equation}\label{MainThmCalc1}
\begin{split}
\hhat_\phi(P) 	& = \frac{1}{|\Zcal_n||\Pcal|}\sum_{Z\in\Zcal_n}\sum_{Q\in \Pcal}(\hhat_\phi(Z) + \hhat_\phi(Q)) \\ 
		& = \frac{1}{[L:K]}\sum_{v\in M_K}\sum_{\sigma:L\hookrightarrow\CC_v}\Big\{\frac{1}{|\Zcal_n||\Pcal|}\sum_{Z\in\Zcal_n}\sum_{Q\in \Pcal}g_{\phi,v}(\sigma(Z),\sigma(Q))\Big\} \\ 
		& = \sum_{v\in M_K}\Big\{\frac{1}{|\Zcal_n||\Pcal|}\sum_{Z\in\Zcal_n}\sum_{Q\in \Pcal}g_{\phi,v}(\epsilon_v(Z),\epsilon_v(Q))\Big\}, 
\end{split}
\end{equation}
which implies $(\ref{Contradiction1})$, by the definition of $\Gamma_v(n)$ and $\Gamma(n)$.  The first identity in $(\ref{MainThmCalc1})$ uses the fact that each $Z\in\Zcal_n$ is preperiodic and therefore $\hhat_\phi(Z)=0$, and that $\hhat_\phi(Q)=\hhat_\phi(P)$ for all $Q\in\Pcal$; the second identity in $(\ref{MainThmCalc1})$ is by Proposition \ref{HeightProp} {\bf (c)};  the third identity in $(\ref{MainThmCalc1})$ follows from the fact that the sets $\Zcal_n$ and $\Pcal$ are $\Gal(\Kbar/K)$-invariant, and thus for each $v\in M_K$ the summand does not depend on the embedding $\sigma:L\hookrightarrow\CC_v$.

We now turn to the proof of $(\ref{Contradiction2})$.  First, fix a place $v\notin S$.  By the hypothesis that each $Z\in \Zcal_v$ is $S$-integral with respect to $P$, we have $r_v(\epsilon_v(Z))\neq r_v(\epsilon_v(Q))$ for each $Z\in \Zcal_v$ and $Q\in\Pcal$, where $r_v:\PP^1(\CC_v)\to\PP^1(k_v)$ is the reduction map; thus 
\begin{equation*}
g_{\phi,v}(\epsilon_v(Z),\epsilon_v(Q)) = -\log \delta_v(\epsilon_v(Z),\epsilon_v(Q)) = 0
\end{equation*}
since $\phi$ has good reduction at $v$ (Proposition \ref{HeightProp} {\bf (b)}).  It follows that $\Gamma_v(n)=0$ for all $v\notin S$.

Now consider a place $v\in S$.  We have
\begin{equation}\label{SigmaCalc}
\Gamma_v(n) = \frac{1}{|\Zcal_n|}\sum_{Z\in\Zcal_n}F(\epsilon_v(Z)),
\end{equation}
where
\begin{equation}\label{FDef}
\begin{split}
F(R) & = \frac{1}{|\Pcal|}\sum_{Q\in\Pcal}g_{\phi,v}(R,\epsilon_v(Q)).
\end{split}
\end{equation}
We would like to calculate the limiting value of $\Gamma_v(n)$ using $(\ref{SigmaCalc})$ and the equidistribution theorem, with $F$ as a testing function; however, the fact that $F$ has logarithmic singularities at the points of $\epsilon_v(\Pcal)$ means that it does not quite meet the hypotheses of Theorem \ref{EquidistThm}.  But $\epsilon_v(\Pcal)$ is contained in the $v$-adic Fatou set $\Fcal_v(\phi)$ by hypothesis, and therefore by Proposition \ref{DiscThm} no point in $\epsilon_v(\Pcal)$ is preperiodic or a limit of preperiodic points.  Moreover, each point in $\epsilon_v(\Pcal)$ is Fatou, and so by $(\ref{SuppFatou})$ it does not meet the support of the canonical measure $\mu_{\phi,v}$.  Since $\epsilon_v(\Pcal)$ is finite, it follows that $F$ can be replaced with a truncation $F^*$ meeting the following two requirements:
\begin{itemize}
	\item $F^*$ is bounded and continuous on $\Psf^1_v$; and
	\item $F^*(R)=F(R)$ for all preperiodic $R$ and for all $R$ in the support of $\mu_{\phi,v}$.
\end{itemize}
In particular $\int F^* d\mu_{\phi,v}=\int F d\mu_{\phi,v}=0$, by Proposition \ref{HeightProp} {\bf (a)}.  We deduce from Theorem \ref{EquidistThm} that
\begin{equation}\label{SigmaTruncCalc}
\Gamma_v(n) = \frac{1}{|\Zcal_n|}\sum_{Z\in\Zcal_n}F^*(\epsilon_v(Z)) \to \int F^* d\mu_{\phi,v} =0
\end{equation}
as $n\to+\infty$.

Finally, by $(\ref{SigmaTruncCalc})$ and the fact that $\Gamma_v(n)=0$ for all $v$ outside the fixed, finite set $S$, we have
\begin{equation*}
\lim_{n\to+\infty}\Gamma(n) = \lim_{n\to+\infty}\sum_{v\in S}\Gamma_v(n) = \sum_{v\in S}\{\lim_{n\to+\infty}\Gamma_v(n)\} = 0.
\end{equation*}
This proves $(\ref{Contradiction2})$, and completes the proof of the theorem.
\end{proof}


\medskip

\end{document}